\documentclass{article}

\usepackage{arxiv}

\usepackage[utf8]{inputenc} % allow utf-8 input
\usepackage[T1]{fontenc}    % use 8-bit T1 fonts
\usepackage{url}            % simple URL typesetting
\usepackage{booktabs}       % professional-quality tables
\usepackage{nicefrac}       % compact symbols for 1/2, etc.
\usepackage{microtype}      % microtypography
\usepackage{lipsum}
\usepackage{amsmath, amsthm, amscd, amsfonts, amssymb, graphicx, color}
\usepackage[backref,colorlinks=true]{hyperref}

\title{Classification of Randers metric of isotropic projective Ricci curvature}

\author{
  Pejhman Vatandoost-Miandehi $^1$ \\
  Department of Mathematics and Computer Science\\
  Amirkabir University of Technology\\
  Tehran, Iran
   \And
 Masoud Nikokar $^2$
 \\
  Department of Mathematics and Computer Science\\
  Amirkabir University of Technology\\
  Tehran, Iran \\
}

\begin{document}
\maketitle
\newtheorem{theorem}{Theorem}[section]
\newtheorem{lemma}[theorem]{Lemma}
\newtheorem{proposition}[theorem]{Proposition}
\newtheorem{corollary}[theorem]{Corollary}
\newtheorem{question}[theorem]{Question}

\theoremstyle{definition}
\newtheorem{definition}[theorem]{Definition}
\newtheorem{algorithm}[theorem]{Algorithm}
\newtheorem{conclusion}[theorem]{Conclusion}
\newtheorem{problem}[theorem]{Problem}

\theoremstyle{remark}
\newtheorem{remark}[theorem]{Remark}
\numberwithin{equation}{section}
\footnotetext[1]{E-mail:~\texttt{(pejhman.vatandoost@gmail.com)}; \texttt{(pejhman.vatandoost@iran.ir)}}
\footnotetext[2]{E-mail:~\texttt{(m.nikokar@iran.ir)}}
\begin{abstract}
In this paper, the concept of isotropic projective Ricci curvature has been investigated. By classification of Randers metric of isotropic projective Ricci curvature, it is shown that Randers metric of projective Ricci curvature is reversible if and only if it is of square projective Ricci curvature.
\end{abstract}

% keywords can be removed

\section{Introduction}
In 2001, Shen introduced the concept of projective Ricci curvature for a Finsler metric as follows \cite{1}
\begin{equation*}\label{1}
PRic:=Ric+(n-1)\{\bar{S}_{|m}y^m+\bar{S}^2\},
\end{equation*}
where, $S$ denotes the curvature $S$ of a non-Riemannian quantity and plays an important role in Finsler geometry, and $Ric$ represents the Ricci curvature \cite{1}. Also, symbol $"|"$ represents the horizontal covariant derivative with respect to the Berwald connection, and
\begin{equation*}
\bar{S}:=\frac{1}{n+1}S.
\end{equation*}	

Cheng et al. \cite{2} expressed an equation for a projective Ricci curvature as follows
\begin{equation*}\label{2}
PRic=Rci+\frac{n-1}{n+1}S_{|m}y^m+\frac{n-1}{(n+1)^2}S^2,
\end{equation*}
and fully classified Randers metrics of a projective Ricci flat curvature \cite{2}. A Finsler metric is said to be from a projective Ricci flat curvature if $PRic=0$.

$(\alpha, \beta)$-metrics form an important and specific group of Finsler metrics that are defined as $F=\alpha \varphi(s)$ where $S=\frac{\beta}{\alpha}$ and   
\begin{equation*}
\alpha=\alpha(x,y)=\sqrt{a_{ij}(x)y^iy^j},
\end{equation*}
is a Riemannian metric and
\begin{equation*}
\beta=\beta(x,y)=b_i(x)y^i,
\end{equation*}
is a $1$-form on $M$ and $\varphi(s)$ is a positive function of class $C^\infty$. Assuming $\varphi(s)=1+s$, then function $F=\alpha+\beta$ which is a Finsler metric given conditions on the $1$-form $\beta$, is called a Randers metric. We Put
\begin{equation*}
r_{ij}:=\frac{1}{2}(b_{i;j}+b_{j;i}),~~~ s_{ij}:=\frac{1}{2}(b_{i;j}-b_{j;i}),
\end{equation*}
where notation ";" represents horizontal covariant derivative with respect to Levi-Civita connection related to metric $\alpha$. We also assume
\begin{equation*}
{r^i}_j:=a^{im}r_{mj},\quad {s^i}_j:=a^{im}s_{mj},\quad r_j:=b^m r_{mj},\quad r:=r_{ij}b^i b^j=b^j r_{j},
\end{equation*}
\begin{equation*}
s_j:=b^m s_{mj},\quad q_{ij}:=r_{im}{s^m}_j,\quad t_{ij}:=s_{im}{s^m}_j,\quad
q_j:=b^i q_{ij}=r_m{s^m}_j,
\end{equation*}
\begin{equation*}
t_j:=b^i t_{ij}=s_m{s^m}_j,\quad b:=\lVert \beta_x\rVert_\alpha,\quad\rho:=\ln\sqrt{1-b^2},\quad\rho_i:=\rho_{x^i},
\end{equation*}
\begin{equation*}
(a^{ij}):=(a_{ij})^{-1},\quad b^i:=a^{ij}b_j,
\end{equation*}
where
\begin{equation*}
r_{i0}:=r_{ij}y^j,\quad s_{i0}:=s_{ij}y^j,\quad
r_{00}:=r_{ij}y^i y^j,\quad
r_0:=r_iy^i,\quad s_0:= s_i y^i,\quad \rho_0:=\rho_i y^i.
\end{equation*}

\begin{definition}\label{def}
	Let $F$ be a Finsler metric on a manifold $M$, and $PRic$ represents the projective Ricci curvature of the metric $F$ with respect to the Busemann-Hausdorff volume form. In this case, $F$ is of an isotropic projective Ricci curvature, if
	\begin{equation*}\label{3}
	PRic=(n-1)cF^2,
	\end{equation*}
	where $c=c(x)$ is a scalar function on $M$. $F$ is said to be of a constant projective Ricci curvature if $c$ is constant. 
\end{definition}

In the following, Randers metrics of an isotropic projective Ricci curvature is classified. In fact, Theorem \ref{thm 1-1} can be expressed and confirmed.
\begin{theorem}\label{thm 1-1}
	Suppose  $F=\alpha+\beta$ is a Randers metric on a manifold $M$. Then, $F$ is of a isotropic projective Ricci curvature is if and only if
	\begin{equation*}\label{4}
	^\propto{Ric}=({t^m}_m+(n-1)c)\alpha^2+2t_{00}+(n-1)[\rho_{0;0}-\rho_0^2+c\beta^2],
	\end{equation*}
	\begin{equation*}\label{5}
	{S^m}_{0;m}=-(n-1)(\rho_m {s^m}_0-c\beta),
	\end{equation*}
	\begin{equation*}\label{6}
	s_0=0,\quad or\quad r_{00}+2\beta s_0=0,
	\end{equation*}
	where symbol  $^\propto{Ric}$ represents the Ricci curvature of a Riemannian metric $\alpha$ and $c=c(x)$ is a scalar function on $M$.
\end{theorem}

\begin{theorem}\label{thm 1-2}
	Let $F=\alpha+\beta$  be a Randers metric on a manifold $M$. Then, $F$ is of a reversible projective Ricci curvature if and only if it is of a square projective Ricci curvature.
\end{theorem}

\begin{corollary}\label{cor 1-3}
	If $F=\alpha+\beta$  is a Randers metric of an isotropic square projective Ricci curvature, then $F$ is Riemannian.
\end{corollary}

We prove Theorem \ref{thm 1-1} in Section \ref{sec 3} and also Theorem \ref{thm 1-2} in Section \ref{sec 4}.

For more information about Finsler Geometry refer to \cite{7,8,9}, and for Randers metric see \cite{10,11,12}.

\section{Preliminaries}
\begin{definition}
	Let $M$ be a differentiable manifold. Then, a Finsler structure on $M$  is a $F:TM \rightarrow [0, \infty)$  mapping which satisfies the following conditions:
	\begin{itemize}
		\item[1.] 	$F$ is smooth on  $TM_0=TM-\{0\}$;
		\item[2.]  	$F$ has positive homogeneity of first degree on $y$, namely for every  $\lambda > 0$,
		\begin{equation*}
		F(x,\lambda y)=\lambda F(x,y);
		\end{equation*}
		\item[3.] 	for every $(x,y)$ from the following matrix $TM$ known as Hessian matrix, it is positive definite:
		\begin{equation*}
		g_{ij}(x,y)=\frac{1}{2}\frac{\partial^2 F^2(x,y)}{\partial y^i\partial y^j},
		\end{equation*}
		that is, for  $X\neq 0$, we have $g(X, X)>0$,  where $g_{ij}$ are elements of tensor $g$. 
	\end{itemize}
	Therefore, $(M,F)$ is called Finsler manifold and $F$ is called Fundamental Finsler function.
\end{definition}

Let $F$ be a Finsler metric on a manifold $M$, then curve $c(t)$ is geodetic in the Finsler manifold $(M,F)$ if the following equation is hold
\begin{equation*}
\ddot{c}(t)+2G^i(c(t),\dot{c}(t))=0,~~~ i \in \{1,\cdots, n \},
\end{equation*}
where $G^i$ is a spray factor obtained from metric $F$ and is defined on $M$ as following
\begin{equation*}
G^i=\dfrac{1}{4}g^{il}\{[F^2]_{x^k y^l}y^k-[F^2]_{x^l}\},
\end{equation*}
where $(g^{ij}):=(g_{ij})^{-1}, y\in T_xM$.

The Riemann curvature is a family of linear mappings on the tangent space defined as follows
\begin{equation*}
R=\{R_y:T_pM\rightarrow T_p M\mid  y\in T_pM,p\in M\},
\end{equation*}
where $R_y$ can be locally expressed with respect to $T_p M$ space bases as follows
\begin{equation*}
R_y={R^i}_k\frac{\partial}{\partial x^i}\otimes dx^k.
\end{equation*}
So that ${R^i}_k={R^i}_k(x,y)$  represents the factors of a Riemann curvature $F$ and is defined by
\begin{equation*}
{R^i}_k=2\frac{\partial G^i}{\partial x^k}-y^j\frac{\partial^2 G^i}{\partial x^j\partial y^k}+2G^j\frac{\partial^2 G^i}{\partial y^j\partial y^k}-\frac{\partial G^i}{\partial y^j}\frac{\partial G^j}{\partial y^k}.
\end{equation*}

The Ricci curvature is the result of the Riemann curvature and is defined as follows
\begin{equation*}
Ric(x,y)={R^i}_i(x,y).
\end{equation*}
By definition, the Ricci curvature is a positive definite function of degree $2$ on $y$ \cite{3}.

In Finsler geometry, there are two important volume forms, one being the Busemann-Hausdorff volume form and the other being the Holmes-Thomson volume form. Suppose $\{e_i\}_{i=1}^n$  is an arbitrary basis for  $T_x M$ and  $\{\theta^i\}_{i=1}^n$ is a dual basis for  $T_x^*M$. Then, bounded open subset $B^n_x$  in $R^n$ is defined as
\begin{equation*}
B_x^n=\{(y^i)\in R^n\mid F(y^ie_i)<1\}.
\end{equation*}

Let $dV_F=\sigma_F(x)\theta^1\wedge\cdots\wedge\theta^n$  is an arbitrary volume form. In this case, supposing $\sigma_F (x)$ as 
\begin{equation*}
\sigma_F(x)=\frac{Vol(B^n(1))}{Vol({B^n}_x)},
\end{equation*}
results in a $dV_f$ volume form called Busemann-Hausdorff volume form where $Vol$ represent Euclidean volume and  $\omega_n$ is the Euclidean volume of unit sphere $B^n$ in $R^n$ \cite{4}. Considering this volume form, the notion of distortion is defined as
\begin{equation*}
\tau(x,y)=\ln\frac{\sqrt{\det(g_{ij}(x,y))}}{\sigma_{BH}(x)}.
\end{equation*}

Clearly, the notion of distortion has a positive homogeneity of zero degree. For a vector $y\in T_x M-\{0\}$, let $c=c(t)$  be geodesic under the conditions $c(0)=x$ and   $\dot{c}(0)=y$. Then, curvature $S$ is defined as
\begin{equation*}
S(x,y)=\frac{d}{dt}[\tau (c(t),\dot{c}(t))]\mid t=0.
\end{equation*}

According to the above definition, curvature $S$, in fact, is a restriction of derivative of $\tau$ on the geodesics, i.e.,
\begin{equation*}
S(x,y)=\tau_{\mid 1}(x,y)y^1.
\end{equation*}

It can be clearly seen that the curvature $S$ has a homogeneity of first degree. That is, for every  $\lambda>0$,
\begin{equation*}
S(x,\lambda y)=\lambda S(x,y).
\end{equation*}
On the other hand, in local coordinates we have
\begin{equation*}
S(x,y)=y^i\frac{\partial\tau}{\partial x^i}-2\frac{\partial\tau}{\partial y^i}G^i.
\end{equation*}
Given this equation, another equation can be obtained for the curvature $S$ as follows
\begin{equation*}
S(x,y)=\frac{\partial G^m}{\partial y^m}-y^m\frac{\partial}{\partial x^m}[\ln\sigma_{BH}].
\end{equation*}

\section{Isotropic projective Ricci curvature}\label{sec 3}
Cheng et al. \cite{2} described a Randers metrics of projective Ricci curvature  as
\begin{align}\label{7}
\aligned
PRic&=~ ^\alpha{Ric}+2\alpha {s^m}_{0;m}-2t_{00}-\alpha^2 {t^m}_m\\
&\quad+(n-1)\{-\frac{2\alpha \beta}{F^2}s_0^2+2\alpha(\rho_m {s^m}_0)-\rho_{0;0}-\frac{\alpha}{F^2}r_{00}s_0+\rho_0^2\}.
\endaligned
\end{align}

In the following we prove Theorem \ref{thm 1-1}.

\begin{proof}[Proof of Theorem \ref{thm 1-1}]
	To prove the necessary condition of Theorem \ref{thm 1-1}, suppose $F$ is of the isotropic projective Ricci curvature. According to \eqref{7}, it results
	\begin{align*}\label{8}
	\aligned
	0&=~ ^\alpha{Ric}+2\alpha {s^m}_{0;m}-2t_{00}-\alpha^2 {t^m}_m\\
	&\quad+(n-1)\{-\frac{2\alpha \beta}{F^2}s_0^2+2\alpha(\rho_m {s^m}_0)-\rho_{0;0}-\frac{\alpha}{F^2}r_{00}s_0+\rho_0^2-F^2c\}.
	\endaligned
	\end{align*}
	Multiplying both sides of the equation in $F^2$ results in 
	\begin{align*} %\label{9}
	\aligned
	0&=F^2~^{\alpha}Ric+2F^2\alpha {s^m}_{0;m}-2F^2t_{00}-\alpha^2F^2 {t^m}_m\\
	&\quad+(n-1)\{-2\alpha \beta s_0^2+2F^2\alpha(\rho_m {s^m}_0)-{F^2}\rho_{0;0}-\alpha r_{00}s_0+ F^2\rho_0^2-F^4 c\}.
	\endaligned
	\end{align*}
	The ration can be rewritten as 
	\begin{equation*}\label{10}
	E_4\alpha^4+E_3\alpha^3+E_2\alpha^2+E_1\alpha+E_0=0,
	\end{equation*}
	where
	\begin{equation}\label{11}
	E_4=-{t^m}_m-(n-1)c,
	\end{equation}
	\begin{equation*}\label{12}
	E_3=2\{{s^m}_{0;m}-\beta {t^m}_m+(n-1)(\rho_m {s^m}_0-2c\beta)\},
	\end{equation*}
	\begin{align}\label{13}
	\aligned
	E_2&=~^\alpha{Ric}+4\beta {s^m}_{0;m}-2t_{00}-\beta^2 {t^m}_m+4(n-1)\beta(\rho_m{s^m}_0)\\&\quad-(n-1)\rho_{0;0}+(n-1)\rho^2_0-6(n-1)c\beta^2,
	\endaligned
	\end{align}
	\begin{align*}\label{14}
	\aligned
	E_1&=2\beta~^\alpha Ric+2\beta^2 {s^m}_{0;m}-4\beta t_{00}-2(n-1)\beta s^2_0+2(n-1)\beta^2(\rho_m {s^m}_0)\\&\quad -2(n-1)\beta\rho_{0;0}+2(n-1)\beta\rho^2_0-(n-1)r_{00}s_0-4(n-1)c\beta^3,
	\endaligned
	\end{align*}
	\begin{equation*}\label{15}
	E_0=(~^\alpha{Ric}-2t_{00}-(n-1)\rho_{0;0}+(n-1)\rho_{0}^{2}-(n-1)c\beta^2)\beta^2.
	\end{equation*}
	From \eqref{11}, we have
	\begin{equation}\label{16}
	E_4\alpha^4+E_2\alpha^2+E_0=0,
	\end{equation}
	\begin{equation}\label{17}
	E_3\alpha^2+E_1=0.
	\end{equation}
	From \eqref{16}, we get 
	\begin{equation}\label{18}
	(E_4\alpha^2+E_2)\alpha^2+E_0=0.
	\end{equation}
	Since  $\alpha^2$ and $\beta^2$  are mutual prime polynomials, from equation \eqref{18} and using $E_0$, it can be said that there is a scalar function $\lambda$ on $M$ such that
	\begin{equation}\label{19}
	^\alpha{Ric}-2t_{00}-(n-1)\rho_{0;0}+(n-1)\rho_0^2-(n-1)c\beta^2=\lambda(x)\alpha^2.
	\end{equation}
	By substituting \eqref{19} in \eqref{18} we have
	\begin{equation}\label{20}
	E_4\alpha^2+E_2+\lambda(x)\beta^2=0.
	\end{equation}
	We also have from \eqref{13}
	\begin{equation}\label{21}
	E_2=\lambda(x)\alpha^2+4\beta {s^m}_{0;m}-\beta^2{t^m}_m+4(n-1)\beta(\rho_m {s^m}_0)-5(n-1)c\beta^2.
	\end{equation}
	We rewrite \eqref{19} as \eqref{22}
	\begin{equation}\label{22}
	^\alpha{Ric}=\lambda(x)\alpha^2+2t_{00}+(n-1)[\rho_{0;0}-\rho_0^2+c\beta^2].
	\end{equation}
	By substituting \eqref{21} in \eqref{20} and using \eqref{11} we have
	\begin{equation*}\label{23}
	[\lambda-{t^m}_m-(n-1)c](\alpha^2+\beta^2)=-4\beta[{s^m}_{0;m}+(n-1)(\rho_m{s^m}_0)-(n-1)c\beta].
	\end{equation*}
	As a result, we have the following equations from above equation
	\begin{equation}\label{24}
	\lambda={t^m}_m+(n-1)c,
	\end{equation}
	\begin{equation}\label{25}
	{s^m}_{0;m}=-(n-1)(\rho_m{s^m}_0-c\beta).
	\end{equation}
	In addition, using the \eqref{22}, \eqref{24} and \eqref{25} we have
	\begin{equation*}\label{26}
	E_1=2[{t^m}_m+(n-1)c]\alpha^2\beta-(n-1)s_0(r_{00}+2\beta s_0),
	\end{equation*}
	\begin{equation*}\label{27}
	E_3=-2\beta[{t^m}_m+(n-1)c].
	\end{equation*}
	As a result, we get from \eqref{17}
	\begin{equation*}\label{28}
	s_0(r_{00}+2\beta s_0)=0,
	\end{equation*}
	And it results from the equation that $S_0=0$ or $r_{00}+2\beta s_0=0$.
	
	The proof of the sufficiency condition is obvious because by substituting three existing conditions from Theorem \ref{thm 1-1} in \eqref{7}, we get
	\begin{equation*}
	PRic=(n-1)cF^2.
	\end{equation*}
	Then, $F$ can be said to be of a isotropic projective Ricci curvature.
\end{proof}

\begin{corollary}\label{Cor 2-1}
	Let $F=\alpha+\beta$  be a Randers metric on a manifold $M$, then $F$ is a projective Ricci flat curvature if and only if the following equations exist
	\begin{equation*}
	^\alpha{Ric}={t^m}_m\alpha^2+2t_{00}+(n-1)[\rho_{0;0}-\rho_{0}^2],
	\end{equation*}
	\begin{equation*}
	{s^m}_{0;m}=-(n-1)(\rho_m {s^m}_0),
	\end{equation*}
	\begin{equation*}
	s_0=0\quad or\quad r_{00}+2\beta s_0=0.
	\end{equation*}
\end{corollary}

\section{Randers metrics of reversible projective Ricci carvature}\label{sec 4}
In this section we rewrite Theorem \ref{thm 1-2} and prove it.

Assume $F=\alpha+\beta$ is of reversible projective Ricci curvature, that is
\begin{equation*}
PRic(y)=PRic(-y).
\end{equation*}
Then the following theorem can be stated.
\begin{theorem}\label{thm 3-1}
	Let $F=\alpha+\beta$ be a Randers metric on a manifold $M$. Then, $F$ is a reversible projective Ricci curvature if and only if
	\begin{equation}\label{29}
	{s^m}_{0;m}=-(n-1)(\rho_m {s^m}_0),
	\end{equation}
	\begin{equation}\label{30}
	r_{00}+2\beta s_0=0,\quad or\quad s_0=0.
	\end{equation}
	In this case, $F$ is of a square projective Ricci curvature.
\end{theorem}

\begin{proof}
	Let $F$ be of a reversible projective Ricci curvature, then using the \eqref{7} and
	\begin{equation*}
	PRic(y)=PRic(-y),
	\end{equation*}
	it results
	\begin{equation*}\label{31}
	4F^2\alpha {s^m}_{0;m}+2(n-1)\{-2\alpha\beta s_0^2+2F^2\alpha(\rho_m {s^m}_0)-\alpha r_{00}s_0\}=0,
	\end{equation*}
	which is equivalent to \eqref{32}
	\begin{equation}\label{32}
	N_3\alpha^3+N_2\alpha^2+N_1\alpha=0,
	\end{equation}
	where
	\begin{align*}
	N_3&=4{s^m}_{0;m}+4(n-1)(\rho_m {s^m}_0), \\
	N_2&=8\beta {s^m}_{0;m}+8(n-1)\beta(\rho_m {s^m}_0), \\
	N_1&=4\beta^2 {s^m}_{0;m}-4(n-1)\beta s_0^2+4(n-1)\beta^2(\rho_m {s^m}_0)-2(n-1)r_{00}s_0. 
	\end{align*}
	So from \eqref{32} we have
	\begin{equation}\label{36}
	N_3\alpha^2+N_1=0
	\end{equation}
	\begin{equation}\label{37}
	N_2=0.
	\end{equation}
	From \eqref{36}, we obtain \eqref{38}
	\begin{equation}\label{38}
	{s^m}_{0;m}=-(n-1)(\rho_m {s^m}_0).
	\end{equation}
	The above equation is equal to the \eqref{29}. By substituting \eqref{38} to \eqref{37} we have
	\begin{equation}\label{39}
	s_0(r_{00}+2\beta s_0)=0.
	\end{equation}
	Consequently $s_0=0$ or $r_{00}+2\beta s_0=0$.
	
	The proof of the sufficiency condition is trivial, because, by assuming that \eqref{29} and \eqref{30} are hold and by substituting equations \eqref{38} and \eqref{39} in \eqref{7} we will have
	\begin{equation*}\label{40}
	PRic=Ric-2t_{00}-\alpha^2 {t^m}_m+(n-1)\{-\rho_{0;0}+\rho_0^2\}.
	\end{equation*}
	It can be clearly seen from the above equation that $F$ is of a reversible projective Ricci curvature. In fact, $F$ is of a square projective Ricci curvature. In this way, the sufficiency condition is confirmed.
\end{proof}

Let Finsler metric $F$ be of the isotropic curvature S, that is
$$S=(n+1)cF,$$
where, $c=c(x)$ is a scalar function on the manifold $M$. Then, we have
\begin{equation*}
S_{\mid m}=(n+1)c_m F,
\end{equation*}
\begin{equation*}
PRic=Ric+(n-1)c_0 F+(n-1)c^2 F^2,
\end{equation*}
where, $c_m:=c_{x^m}$  and  $c_0:=c_m y^m$. In this case, $F$ is of the square projective Ricci curvature if and only if
\begin{equation}\label{41}
Ric_{.j.k.l}+(n-1)\{c_0 F_{y^j y^k y^l}+c^2F^2_{y^j y^k y^l}\}=0.
\end{equation}
Using this equation, Theorem \ref{thm 3-2} can be expressed.
\begin{theorem}\label{thm 3-2}
	Let  $F=\alpha+\beta$ be a Randers metric on a $n$-dimensional manifold $M$ and is of a square Ricci curvature, then $F$ is of a square projective Ricci curvature if and only if $S=0$.
\end{theorem}

\begin{proof}
	The sufficiency condition is trivial. We prove the necessary condition. Suppose $F=\alpha+\beta$  is a Randers metric, then
	\begin{equation}\label{42}
	F_{y^jy^ky^l}=-\frac{1}{\alpha^3}\left[\delta_{jk}~y_{l}(j\rightarrow k\rightarrow l\rightarrow j)+\frac{3}{\alpha^2}y_l y_k y_j\right],
	\end{equation}
	and
	\begin{align}\label{43}
	\aligned
	F^2_{y^jy^ky^l}&=\frac{1}{\alpha}b_j\delta_{kl}(j\rightarrow k\rightarrow l\rightarrow j)-\frac{1}{\alpha^3}\beta\delta_{jk}y_l(j\rightarrow k\rightarrow l\rightarrow j)\\&\quad-\frac{3}{\alpha^5}\beta y_ly_ky_j-\frac{1}{\alpha^3}b_jy_ky_l(j\rightarrow k\rightarrow l\rightarrow j),
	\endaligned
	\end{align}
	where symbol $j\rightarrow k\rightarrow l\rightarrow j$  represents all of the rotations on the indices and then the summation of them. Suppose $F$ is a square Ricci curvature. Then, we have
	\begin{equation}\label{44}
	Ric_{.j.k.l}=0.
	\end{equation}
	By substituting \eqref{42}, \eqref{43} and \eqref{44} in \eqref{41} we have
	\begin{align}\label{45}
	\aligned
	0=(n-1)\Bigg\{&-\frac{1}{\alpha^3}\left[(c_0+\beta c^2)\delta_{jk}~ y_l(j\rightarrow k\rightarrow l\rightarrow j)+\frac{3}{\alpha^2}(c_0+\beta c^2)y_ly_ky_j\right]\\
	&+\frac{c^2}{\alpha}\left[\left(\delta_{kl}-\frac{1}{\alpha^2}y_ky_l\right)b_j\right](j\rightarrow k\rightarrow l \rightarrow j)\Bigg\}.
	\endaligned
	\end{align}
	It is clearly seen that \eqref{45} exists when $c=0$. Therefore $S=0$.
\end{proof}


\begin{thebibliography}{widest-label}
	
	\bibitem{1}
	Z. M. Shen, \textit{Differential Geometry of Spray and Finsler Spaces}, Kluwer Academic Publishers, Dordrecht, (2001).
	
	\bibitem{2} 
	X. Y. Cheng, Y. L. Shen, X. Y. Ma, \textit{On a class of projective Ricci flat Finsler metrics}, Publ. Math. Debrecen, \textbf{90} (2017), 169--180.
	
	\bibitem{3}
	A. Tayebi, M. Shahbazi Nia, \textit{A new class of projectively flat Finsler metrics with constant flag curvature $\textbf{K}=1$},
	Differential Geom. Appl., \textbf{41} (2015), 123--133.
	
	\bibitem{4}
	Z. M. Shen, \textit{Volume comparison and applications in Riemann-Finsler geometry}, Adv. Math., \textbf{128} (1997), 306--328.
	
	%\bibitem{5} A. Tayebi, H. Sadeghi, \textit{Generalized P-reducible $(\alpha,\beta)$-metrices with vanishing S-curvature}, Ann. Polon. Math., \textbf{114} (2015), 67--79.
	
	%\bibitem{6} A. Tayebi, M. Rafie-Rad, \textit{$S$-curvature of isotropic Berwald metrics}, Sci. China Ser. A, \textbf{51} (2008), 2198--2204.
	
	\bibitem{7}
	Z. M. Shen, \textit{Lectures on Finsler geometry}, World Scientific Publishing Co., Singapore, (2001).
	
	\bibitem{8}
	D. Bao, S.-S. Chern, Z. Shen, \textit{An introduction to Riemann-Finsler geometry}, Springer-Verlag, New York, (2000).
	
	\bibitem{9}
	H. Rund, \textit{The differential geometry of Finsler spaces}, Springer-Verlag, Berlin, (2012).
	
	\bibitem{10}
	X. Y. Chen, Z. M. Shen, \textit{Randers metrics with special curvature properties}, Osaka J. Math., \textbf{40} (2003), 87--101.
	
	\bibitem{11}
	Z. M. Shen, H. Xing, \textit{On Randers metrics with isotropic $S$-curvature}, Acta Math. Sin. (Engl. Ser.), \textbf{24} (2008), 789--796.
	
	\bibitem{12}
	X. Y. Cheng, Z. M. Shen, \textit{Randers metrics of scalar flag curvature}, J. Aust. Math. Soc., \textbf{87} (2009), 359--370.
	
	
	
\end{thebibliography}
\end{document}